\documentclass[11pt]{amsart}
\usepackage{fullpage}

\newtheorem{theorem}{Theorem}[section]
\newtheorem{example}[theorem]{Example}
\newtheorem{lemma}[theorem]{Lemma}
\newtheorem{corollary}[theorem]{Corollary}
\newtheorem{remark}[theorem]{Remark}
\newtheorem{proposition}[theorem]{Proposition}
\theoremstyle{definition}
\newtheorem{definition}[theorem]{Definition}

\begin{document}
\title{Graded algebras with homogeneous involution and varieties of almost polynomial growth}

\author{Wesley Quaresma Cota$^*$}
\address{Department of Mathematics, Instituto de Matem\'atica, Estat\'istica e Ci\^encia da Computa\c c\~ao, Universidade de S\~ao Paulo, SP, Brazil}
\email{quaresmawesley@gmail.com}

\author{Felipe Yasumura}
\address{Department of Mathematics, Instituto de Matem\'atica, Estat\'istica e Ci\^encia da Computa\c c\~ao, Universidade de S\~ao Paulo, SP, Brazil}
\email{fyyasumura@ime.usp.br}
\thanks{{\it E-mail addresses:} quaresmawesley@gmail.com (Cota), fyyasumura@ime.usp.br (Yasumura)}
\thanks{\footnotesize W. Q. C.~is partially supported by FAPESP, grant no.~2025/05699-0.}
\thanks{F. Y.~is supported by FAPESP, grant no.~2024/14914-9.}
\thanks{$^{*}$ Corresponding author}



\dedicatory{Instituto de Matemática, Estatística e Ciência da Computação, Universidade de São Paulo, São Paulo,
Brazil}

\begin{abstract}  
An important aspect in the theory of algebras with polynomial identities is the study of the asymptotic behavior of the  codimension sequence $c_n(A),\, n\geq 1,$ which measures the growth of polynomial identities of a given algebra $A$. In this context, graded identities naturally arise as prominent tools, since ordinary polynomial identities can be viewed as a particular case of graded identities. Moreover, as an involution does not necessarily preserve the homogeneous components of a grading, it is natural to consider the notion of a homogeneous involution. In this work, we investigate the behavior of the codimension sequence in the setting of $G$-graded algebras endowed with a homogeneous involution. More specifically, we characterize the varieties of polynomial growth in terms of the exclusion of a list of algebras from the variety. As a consequence, we provide the classification of the varieties with almost polynomial growth in this setting.
\end{abstract}

\maketitle

\section{Introduction}

The systematic study of group graded algebras was initiated in \cite{PZ89}, where the authors were interested in the classification of group gradings on simple Lie algebras and their applications. Addressing this problem requires, in particular, the classification of degree-preserving involutions on simple associative algebras graded by a group \cite{BSZ} (see also \cite{EK13} and the references therein). Since then, several works have been devoted to the study of polynomial identities of graded algebras equipped with a degree-preserving involution.

On the other hand, in many natural situations where an algebra is endowed with a group grading and an involution, the involution acts by inverting degrees. Properties of such degree-inverting involutions have been examined, for example, in \cite{Hazrat}. Furthermore, when studying polynomial identities of algebras equipped with a group grading and a compatible involution, the degree-inverting case often arises as the most natural setting (see, for instance, \cite{Mello1}).

Motivated by these examples, T.~de Mello \cite{Mello} introduced the notion of a homogeneous involution, that is, an involution on a graded algebra that maps each homogeneous component to another (possibly different) homogeneous component. This framework provides a unified treatment of degree-preserving, degree-inverting, and intermediate types of involutions. Moreover, it is particularly well suited for the study of polynomial identities of graded algebras endowed with a compatible involution. The works \cite{MY,Y} further develop this line of investigation.

In another direction, the study of the growth of polynomial identities of a PI-algebra $A$ was introduced by Regev. The author proposed an approach based on analyzing the sequence of dimensions of the spaces $P_n$ of multilinear polynomials of degree $n$ modulo the corresponding multilinear identities, known as the {codimension sequence} $c_n(A),\, {n \in \mathbb{N}}$. 
Since, in characteristic zero, every polynomial identity of $A$ follows from its multilinear ones, this sequence effectively measures the growth of the polynomial identities of $A$.

In his seminal work~\cite{RG}, Regev proved that the codimension sequence of any PI-algebra is exponentially bounded. Consequently, one can consider the limit $$\textnormal{exp}(A)=\lim_{n \to \infty} \sqrt[n]{c_n(A)},$$ called the PI-exponent of $A$. A groundbreaking result of Giambruno and Zaicev~\cite{GiZai} showed that this limit always exists and is a non-negative integer, thus providing an affirmative answer to a question posed by Amitsur.

The notion of varieties of almost polynomial growth is introduced, where the variety has exponential growth and any proper subvariety has polynomial growth. In terms of the exponent, Kemer \cite{Kem} proved that $\textnormal{exp}(A)\geq 2$ if and only if the variety generated by $A$ contains either the Grassmann algebra $\mathcal{G}$ or the algebra $UT_2$ of upper triangular matrices of size $2$. As a consequence, $\mathcal{G}$ and $UT_2$ generate the only varieties of almost polynomial growth. Moreover, the codimension sequence of a PI-algebra either has polynomial growth or exponential growth, with no intermediate behavior.

It is natural to address analogous questions in the context of graded algebras and algebras with involution. In this direction, Valenti~\cite{Valenti} classified the varieties of $G$-graded algebras with almost polynomial growth. Subsequently, Giambruno and Mischenko~\cite{GM} obtained an analogous classification in the setting of algebras with involution. More recently, the authors in~\cite{Lorena} investigated this problem for algebras endowed with a degree-preserving involution.

In this paper, we investigate $G$-graded algebras endowed with a homogeneous involution. We characterize the varieties of polynomial growth by describing them through the exclusion of certain algebras, where the corresponding list depends on the type of involution considered: degree-preserving, degree-inverting, or arbitrary. As a consequence, we obtain a classification of the varieties of almost polynomial growth in this setting. Moreover, we conclude that the codimension sequence of a graded algebra with a homogeneous involution is either polynomially bounded or grows exponentially. 

It is worth mentioning that our results extend those previously established for algebras with involution and for $G$-graded algebras with a degree-preserving involution.

\section{Algebras with Homogeneous Involution}

Let $G$ be a finite group (not necessarily abelian) and $A$ an associative algebra over a field $F$ of characteristic zero.

Recall that a linear map $*:A\rightarrow A$ is called an involution if 
\[
(ab)^{*}=b^{*}a^{*}\quad\text{and}\quad (a^{*})^{*}=a,
\]
for all $a,b\in A$. An algebra $A$ endowed with an involution $*$ is called an {algebra with involution} or a $*$-algebra.

For a $*$-algebra $A$, we define $c_n^*(A)$, $n \geq 1$, as the $n$-th $*$-codimension of $A$, which measures the dimension of the space $P_n^*$ of multilinear $*$-polynomials modulo the $T_*$-ideal $\textnormal{Id}^*(A)$ of polynomial identities of $A$.

\begin{example}
    Consider the subalgebra $
M = F(e_{11} + e_{44}) + F(e_{22} + e_{33}) + Fe_{12} + Fe_{34}$ of the algebra $UT_4(F)$ of $4\times 4$ upper triangular matrices. Denote by $M_\rho$ the algebra $M$ equipped with the reflection involution $\rho$, which is defined by reflecting a matrix along its secondary diagonal, that is,
\[
\left(
\begin{array}{cccc}
a & b & 0 & 0 \\
0 & c & 0 & 0 \\
0 & 0 & c & d \\
0 & 0 & 0 & a
\end{array}
\right)^{\rho}
=
\left(
\begin{array}{cccc}
a & d & 0 & 0 \\
0 & c & 0 & 0 \\
0 & 0 & c & b \\
0 & 0 & 0 & a
\end{array}
\right).
\]
\end{example}

Given a group $G$, we say that $A$ is a {$G$-graded algebra} if there exist subspaces $A_{g}$, $g\in G$, called the {homogeneous components of degree $g$}, such that 
\[
A=\bigoplus_{g\in G}A_{g}\quad\text{and}\quad A_{g}A_{h}\subseteq A_{gh}, \mbox{ for all }g,h\in G.
\]

The homogeneous degree of a nonzero element $a\in A_g$ is defined as $\mbox{deg }a=g$.

\begin{example} 
 Let $p$ be a prime number dividing $|G|$ and let $g \in G$ be an element of order $p$.  
Denote by $C_p = \langle g \rangle$ the cyclic subgroup generated by $g$ and consider the group algebra $F C_p$ over the field $F$. The canonical grading on $FC_p$ is given by $(FC_p)_{g^i}= Fg^i$, for all $1\leq i
\leq p$. 
\end{example}

Similarly, for a $G$-graded algebra $A$, we associate to it the sequence $c_n^G(A)$, $n \geq 1$, of $G$-graded codimensions, which measure the dimension of the space $P_n^G$ of multilinear $G$-graded polynomials modulo the $T_G$-ideal $\textnormal{Id}^G(A)$ of $G$-graded polynomial identities of $A$.

The {support} of the $G$-grading on $A$, denoted by $\operatorname{supp}(A)$, is the set of all $g\in G$ such that $A_{g}\neq \{0\}$.

\begin{definition}
Let $A=\bigoplus_{g\in G}A_{g}$ be a $G$-graded algebra endowed with an involution $*$ and let $\tau:\operatorname{supp}(A)\rightarrow \operatorname{supp}(A)$ be a bijection. The involution $*$ is said to be {homogeneous with respect to} $\tau$ (or simply a {$\tau$-involution}) if  $(A_{g})^{*}\subseteq A_{\tau(g)}$, for all $g\in G$.
\end{definition}

If $*$ is a $\tau$-homogeneous involution, then $\tau(\tau(g))=g$ and $\tau(gh)=\tau(h)\tau(g)$ for all $g$, $h\in\operatorname{supp}(A)$ such that $A_gA_h\ne0$. It does not necessarily imply that $\tau$ can be extended to an involution of the group $G$. However, if we consider the free $G$-graded algebra, then the support of the grading is $G=\operatorname{supp}(A)$; i.e., the existence the free $(G,*)$-algebra (see below) implies that $\tau$ is an involution. Thus, we shall assume that $\tau$ is an involution of the group from now on.

As a particular case of a homogeneous involution, if $*$ is homogeneous with respect to the map $\tau(g)=g^{-1}$ for all $g\in \operatorname{supp}(A)$, then $*$ is called a degree-inverting involution. On the other hand, if $(A_{g})^{*}\subseteq A_{g}$ for all $g\in G$, we say that $*$ is a degree-preserving involution or a graded involution, i.e., an involution with respect to the trivial map $\tau(g)=g$. Note that, in this case, whenever $A_gA_h\neq 0$, we must have $gh=hg$, which imposes a restriction on the subset $\operatorname{supp}(A)$. Consequently, since the above definition constrains the support of the grading, it is natural to adopt a more general notion of homogeneous involution, allowing us to remove this limitation.

We refer to a $G$-graded algebra endowed with a homogeneous involution simply as a {$(G,*)$-algebra}, specifying in the context whether the involution is graded, degree-inverting, or of a general homogeneous type.

For instance, consider $FC_{2,*}$ as the algebra $FC_2$ with trivial grading and involution $(\alpha1 +\beta g)^*=\alpha1 -\beta g$.

\begin{example}  Let $g\in G$ be an element of prime order $p$ such that $\tau(g)\in\{g,g^{-1}\}$ and $C_p=
\langle g \rangle$. Define:

\begin{enumerate}
\item[$(1)$] $FC_{2,*}^{g}$ is the algebra $FC_2$ with canonical grading and involution $(\alpha1 +\beta g)^*=\alpha1 -\beta g$. 
   \item[$(2)$]  $FC_{p,\tau}^{g}$ is the algebra $FC_p$ with canonical grading and $\tau$-involution $g^*=\tau(g)$.
\end{enumerate}
\end{example}

\begin{example}
    Let $g\in G$ and denote by $M_{\rho, \tau}^{g}$ the $(G,*)$-algebra $M$ defined above endowed with reflection involution and $G$-grading given by
$$(M_{\rho,\tau }^{g})_{1}=F(e_{11}+e_{44})+F(e_{22}+e_{33}),\quad (M_{\rho,\tau}^{g})_{g} =F e_{12}$$ $$(M_{\rho, \tau }^{g})_{\tau(g)} =Fe_{34}\quad  \mbox{ and }\quad (M_{\rho,\tau}^{g})_{r} = \{0\}, \mbox{ for all }r\in G- \{1,g,\tau(g)\}.$$
\end{example}

If $A$ and $B$ are $(G,*)$-algebras, a homomorphism $\varphi: A \to B$ is called a {homomorphism of graded algebras with homogeneous involution} (or simply a {$(G,*)$-homomorphism}) if $\varphi(A_g)\subseteq B_g$ for all $g\in G$ and $\varphi(a^*)=\varphi(a)^*$ for all $a\in A$. In addition, if $\varphi$ is bijective, then we say that $\varphi$ is an isomorphism and $A$ and $B$ are isomorphic as $(G,*)$-algebras.

For each $g\in G$, consider the countable sets of noncommuting variables 
\[
X_g=\{x_{k,g}\mid k\in \mathbb{N}\}, \qquad X_g^*=\{x_{k,g}^*\mid k\in \mathbb{N}\}.
\]
Let $X=\bigcup_{g\in G}X_g$, $X^*=\bigcup_{g\in G}X_g^*$, and denote by $\mathcal{F}=F\langle X\cup X^*\rangle$ the free associative algebra generated by $X\cup X^*$. We endow $\mathcal{F}$ with a natural involution
defined by
\[
(x_{i,g_i})^*=x_{i,g_i}^*, \qquad (x_{i,g_i}^*)^*=x_{i,g_i}, \qquad (x_{i,g_i}^{\epsilon_i}x_{j,g_j}^{\epsilon_j})^*=(x_{j,g_j}^{\epsilon_j})^*(x_{i,g_i}^{\epsilon_i})^*,
\]
where $\epsilon_k$ denotes either $*$ or the empty symbol.

Now define a $G$-grading on $\mathcal{F}$ so that it becomes a $G$-graded algebra endowed with a homogeneous involution with respect to a map $\tau:G\to G$. Set $\deg(1)=1$, and for each $g\in G$ and $k\in \mathbb{N}$ define $\deg(x_{k,g})=g$ and $\deg(x_{k,g}^*)=\tau(g)$. For a monomial $y=x_{i_1,g_1}^{\epsilon_1}\cdots x_{i_n,g_n}^{\epsilon_n}\in \mathcal{F}$, we define 
\[
\deg(y)=\deg(x_{i_1,g_1}^{\epsilon_1})\cdots \deg(x_{i_n,g_n}^{\epsilon_n}).
\]
Then, setting
\[
\mathcal{F}_g=\operatorname{span}_F\{y=x_{i_1,g_{i_1}}^{\epsilon_1}\cdots x_{i_n,g_{i_n}}^{\epsilon_n}\mid \deg(y)=g\},
\]
we obtain $\mathcal{F}=\bigoplus_{g\in G}\mathcal{F}_g$, a $G$-grading that makes $\mathcal{F}$ a $G$-graded algebra with a homogeneous involution relative to $\tau$. We call this algebra the {free associative $G$-graded algebra with a homogeneous involution with respect to} $\tau$.

A polynomial $f\in\mathcal{F}$ is said to be a {$(G,*)$-identity} of $A$ if it vanishes under any admissible substitution of variables by elements of $A$, that is, when each $x_{k,g}$ is replaced by an element $a\in A_g$ and the corresponding $x_{k,g}^*$ by $a^*$, for all $k\in\mathbb{N}$. In this case, we write $f\equiv 0$ on $A$.

A two-sided ideal $I\subseteq\mathcal{F}$ is called a {$T_{G}^*$-ideal} if it is invariant under all $(G,*)$-endomorphisms of $\mathcal{F}$. In general, we are interested in studying the set of all $(G,*)$-identities of a given algebra $A$, denoted by
\[
\operatorname{Id}^{\sharp}(A)=\{f\in\mathcal{F}\mid f\equiv 0 \text{ on }A\}.
\]
It is straightforward to verify that $\operatorname{Id}^{\sharp}(A)$ is a $T_G^*$-ideal of $\mathcal{F}$. Hence, we refer to it as the $T_G^*$-ideal of $A$. The {$(G,*)$-variety} generated by $A$, denoted by $\operatorname{var}^{\sharp}(A)$, is the class of all $(G,*)$-algebras $B$ satisfying $\operatorname{Id}^{\sharp}(A)\subseteq \operatorname{Id}^{\sharp}(B)$. We say that $A$ and $B$ are $T_G^*$-equivalent if $\textnormal{var}^\sharp(A)=\textnormal{var}^\sharp(B)$.

We denote by 
\[
P_n^{\sharp } = \operatorname{span}_F\{x_{\sigma(1),g_{1}}^{\epsilon_1}\cdots x_{\sigma(n),g_{n}}^{\epsilon_n}\mid \sigma\in S_n,\, g_{1},\ldots,g_n\in G,\, \epsilon_1,\ldots,\epsilon_n\in\{*,\emptyset\}\}
\]
the space of multilinear $(G,*)$-polynomials of degree $n$. As in the classical case, when $\operatorname{char} F = 0$, $\operatorname{Id}^\sharp(A)$ is generated by its multilinear $(G,*)$-identities.

In this paper, we are concerned with the growth of the $(G,*)$-identities of $A$. For that purpose, we consider the quotient space
\[
P_n^{\sharp}(A)=\frac{P_n^{\sharp}}{P_n^{\sharp}\cap \operatorname{Id}^{\sharp}(A)}
\]
and define $c_n^{\sharp }(A)=\dim_F P_n^{\
\sharp}(A)$ as the $n$th $(G,*)$-codimension of $A$. We define $c_n^\sharp(\mathcal{V})=c_n^\sharp(A)$, where $A$ is a $(G,*)$-algebra satisfying $\mathcal{V}=\textnormal{var}^\sharp(A)$.

If $A$ is a $(G,*)$-algebra, one may study its ordinary identities, $*$-identities, and $G$-graded identities. Moreover, the spaces $P_n, P_n^* $ and $P_n^G$ can be viewed as subspaces of $P_n^\sharp$. The connection between the corresponding codimensions is established in the next lemma, whose proof follows directly from known results in the literature (see \cite{GReg}).

\begin{proposition}\label{desigualdades}
  For a $(G,*)$-algebra $A$ we have
$$c_n(A)
    \leq \,c_n^*(A),\, c_n^G(A)\,
    \leq c_n^\sharp(A)\leq 2^n|G|^nc_n(A).$$
  
\end{proposition}

An important consequence of the previous result and \cite{RG} concerning the sequence of $(G,*)$-codimensions is stated below.

\begin{proposition} \label{exponentiallubounded}
    Let $A$ be a $(G,*)$-algebra satisfying an ordinary non-trivial polynomial identity. Then the $(G,*)$-codimension sequence $c_n^\sharp(A)$, $n = 1,2, \ldots$, is exponentially bounded.
\end{proposition}

In particular, the previous result allows us to conclude the following remark.

\begin{remark} \label{exponentialgrowth}
According to \cite{GM2, MV, Valenti}, the $(G,*)$-varieties generated by $FC_{2,*}$, $M_{\rho,\tau}^{g}$ and $FC_{p,\tau}^{h}$ have exponential growth of their sequences of $(G,*)$-codimensions, for all $g,h\in G$ with $|h|=p$, where $p$ is a prime number dividing $|G|$, $\tau(h)\in \{h,h^{-1}\}$ and $C_p=\langle h\rangle$. Moreover, if $|G|$ is even and $s\in G$ with $|s|=2$ and $\tau(s)=s$, then $FC_{2,*}^s$ has exponential growth.
\end{remark}

Before concluding this section, we discuss some aspects of the Wedderburn–Malcev Theorem in the context of $(G,*)$-algebras, which provides a structural description of finite-dimensional algebras. To this end, we first recall the classification of finite-dimensional $G$-graded algebras obtained in \cite{BaSehZa,Bahturin}.

\begin{lemma}
    \label{simplesgraded}
    Let $A$ be a finite-dimensional graded-simple $G$-graded algebra over an
algebraically closed field $F$. Then, $A\cong F^{\sigma}[H]\otimes M_n(F)$ where $H$ is a finite subgroup of $G$ and $\sigma$ is a $2$-cocycle of $H$. The grading is given by an $n$-tuple $\alpha= (\alpha_1, \ldots, \alpha_n)\in G^n$ such that $\mbox{deg }(e_{g}\otimes e_{ij})=\alpha_i^{-1}g\alpha_j.$
\end{lemma}

Recall that a finite-dimensional $(G,*)$-algebra $A$ is said to be $(G,*)$-simple if $A^2 \neq \{0\}$ and $A$ contains no nonzero proper $(G,*)$-ideals, that is, $G$-graded ideals invariant under the involution.

\begin{proposition} \label{wedderburnmalcev} Let $G$ be a finite group and $A$ be a finite-dimensional $(G,*)$-algebra over a field of characteristic zero. Then:

\begin{enumerate}
    \item[(1)] $J(A)$ is a $(G, *)$-ideal;
    \item[(2)]  If $A$ is a $(G, *)$-simple algebra, then either $A$ is graded-simple or $A=B\oplus B^*$, where $B$ is a graded-simple subalgebra of $A$;
    \item[(3)] $A=S+J(A)$, where $S=B_1\oplus \cdots \oplus B_m$ is a maximal semisimple $(G,*)$-subalgebra of $A$ and $B_1, \ldots, B_m$ are $(G,*)$-simple algebras.
\end{enumerate}
\end{proposition} 
\begin{proof}
(1) Clearly $J(A)$ is a $\ast$-ideal. In addition, by Gordienko's result \cite[Corollary 3.3]{Gord}, it is a graded ideal.

(2) From (1), $J(A)=0$. Thus, $A=B_1\oplus\cdots\oplus B_s$ is a direct sum of graded-simple algebras and is unital. Then, $B_1^*$ is a graded ideal; so $B_1^*=B_j$, for some $j\in\{1,\ldots,s\}$. If $j=1$, then $B_1$ is a $(G,*)$-ideal, thus $A=B_1$ is $(G,*)$-simple. Otherwise, $B_1\oplus B_j$ is a $(G,*)$-ideal, i.e., $A=B_1\oplus B_1^*$.

(3) From (1), $J(A)$ is a graded ideal. Since $\mathrm{char}\,F=0$, there exists a graded Wedderburn-Malcev decomposition $A=S+J(A)$, where $S=C_1\oplus\cdots\oplus C_s$ is a direct sum of graded-simple algebras. From (2), for each $C_i$, either it is $(G,*)$-simple or $C_i\oplus C_i^*$ is $(G,*)$-simple, and $C_i^*$ is among the algebras $C_1$, \dots, $C_s$.
\end{proof}

In the following, we classify all homogeneous involution on $FC_p$ with the canonical grading.

\begin{proposition} \label{Homogeneousinvolutiononfcp}
    Let $g\in G$ be an element of prime order $p$ and $A=FC_p$ the group algebra of $C_p=\langle g\rangle$ endowed with the canonical $G$-grading. Then, $*$ is a homogeneous $\tau$-involution on $FC_p$ if and only if $\tau(g)\in\{g,g^{-1}\}$ and $(FC_p,*)\cong FC^{g}_{p,\tau}$ or $(FC_p,*)\cong FC^{g}_{2,*}$.
\end{proposition}

\begin{proof}
  Since $C_p = \langle g \rangle$ and $*$ is a homogeneous $\tau$-involution on $A=FC_p$ with canonical grading, we have $g^* = \alpha g^i$ for some integer $i$ and some scalar $\alpha \in F$. Applying the involution, we obtain
\[
g = (g^*)^* = (\alpha g^i)^* = \alpha^{\,i+1} g^{i^2}.
\]
Hence $i^2 \equiv 1 \pmod{p}$ and consequently $i = \pm 1$. Since $\deg g^*=\tau(\deg g)$, we get $\tau(g)\in\{g,g^{-1}\}$. Moreover,
\[
1 = (g^*)^p = (\alpha g^i)^p = \alpha^p g^{ip} = \alpha^p1,
\]
which shows that $\alpha$ is a $p$-th root of unity. Thus, if $p=2$, then we get $(FC_2,*)\cong FC_{2,*}^{g}$ in case $g^*=-g$ or $(FC_p,*)\cong FC_{2,\tau}^{g}$ in case $g^*=g$. So, assume that $p>2$.

Let $g^* = \alpha g$. Then
\[
g = (g^*)^* =  \alpha^2 g.
\]
Hence $\alpha^2 = 1$, and consequently $\alpha \in \{1, -1\}$. Since $p$ is odd and $\alpha$ is a primitive $p$-th root of unity, it follows that $\alpha = 1$. Therefore, $g^* = g$ and $A \cong FC_{p,\tau}^{g}$ where $\tau(g)=g$ in this case.

Now, assume that $g^* = \alpha g^{-1}$ and let $\beta$ be such that $\beta^p=1$ and $\beta^2=\alpha$. The graded automorphism given by $\varphi: FC_p\to  FC_p$ determined by $\varphi(g)= \beta^{-1} g$ defines an isomorphism of $(G,*)$-algebras $(FC_p,*)\cong FC_{p,\tau}^{g}$, where $\tau(g)=g^{-1}$. 
\end{proof}

The previous proposition asserts that every homogeneous involution on $F C_p$ is either a degree-preserving involution or a degree-inverting involution. However, some care is required when dealing with $\tau$-involutions. Indeed, if $\tau$ is the identity map, then the algebra $F C_{p,\tau }^{g}$ where $g^*=g^{-1}$, for $|g| > 2$, is not well-defined. A similar phenomenon arises if $\tau$ is the inversion map and one considers the algebra $F C_{p,\tau}^{g}$ with $g^*=g$. Therefore, it is essential to specify the context in which we are operating.

\section{Classifying Varieties of Almost Polynomial Growth}

In this section we classify, up to $T_G^*$-equivalence, the finite-dimensional $(G,*)$-algebras of almost polynomial growth. Once again, we assume that $G$ is a finite group, $F$ is a field of characteristic zero and $\tau$ is an involution on $G$.

\begin{remark}\label{obs}
    Let $A$ be a $(G,*)$-algebra over a field $F$ of characteristic zero and let $\overline{F}$ be the algebraic closure of $F$. Then $A\otimes \overline{F}$ has a structure of $(G,*)$-algebra where $(a\otimes \alpha)^*=a^*\otimes \alpha $ and $(A\otimes \overline{F})_g= A_g\otimes \overline{F}$, for all $a\in A, \alpha \in 
    \overline{F}$ and $g\in G$. Moreover, $\textnormal{Id}^\sharp(A)=\textnormal{Id}^\sharp (A\otimes \overline{F})$ both over $F$ and
    $c_n^
\sharp(A)$ over $F$ is equal to $c_n^\sharp (A\otimes \overline{F})$ computed over $\overline{F}$.

In addition, the algebras $FC_{2,*}$, $FC_{2,*}^g$, $FC_{p,\tau}^g$, and $M_{\rho,\tau}^g$ are defined over the prime field of $F$. In particular,
\[
FC_{2,*}\otimes_F \overline{F} \cong \overline{F}C_{2,*}, \qquad
FC_{p,\tau}^g \otimes_F \overline{F} \cong \overline{F}C_{p,\tau}^g, \qquad
FC_{2,*}^g \otimes_F \overline{F} \cong \overline{F}C_{2,*}^g,
\]
and
\[
M_{\rho,\tau}^g(F)\otimes_F \overline{F} \cong M_{\rho,\tau}^g(\overline{F}).
\]

Furthermore, since $F$ is an infinite field, the condition $B \notin \mathrm{var}_F^\sharp(A)$ implies that $\overline{B} \notin \mathrm{var}_{\overline{F}}^\sharp(\overline{A})$ (note that this might fail if $F$ is finite). Here the subscript $E$ (where $E = F$ or $E = \overline{F}$) indicates that we are considering the variety of $E$-algebras. Indeed, assume that $\overline{B}$ belongs to the variety of $\overline{F}$-algebras generated by $\overline{A}$, and let $f$ be an $F$-polynomial identity of $A$. Since $F$ is infinite, $f$ is also a polynomial identity of $\overline{A}$. By hypothesis it follows that $f$ is an identity of $\overline{B}$. As $B \subseteq \overline{B}$, we conclude that $f$ annihilates $B$ as well. Thus $B$ lies in the variety of $F$-algebras generated by $A$.

In summary, given an infinite field $F$ and $F$-algebras $A$ and $B$, where $B$ is defined over the prime field of $F$, $B \notin \mathrm{var}_F^\sharp(A)$ implies $\overline{B} \notin \mathrm{var}_{\overline{F}}^\sharp(\overline{A})$. Moreover, $c_n^\sharp(A) = c_n^\sharp(\overline{A})$, where the right-hand side is computed considering $\overline{A}$ as an algebra over $\overline{F}$.

\end{remark}

According to the previous remark, from now on we assume that $F$ is an algebraically closed field.

\begin{lemma} \label{subalgebraofA}
    Let $1<H\leq G$ and $A=F^\sigma H$ be a $(G,*)$-algebra with $\tau$-involution. Then $A$ contains a $(G,*)$-subalgebra isomorphic to $FC_{p,\tau}^{h}$ or $FC_{2,*}^h$, for some $h\in H$ of prime order $p$ generating $C_p$ and $\tau(h)\in \{h,h^{-1}\}$.
\end{lemma}
\begin{proof}
From Proposition~\ref{Homogeneousinvolutiononfcp}, it is enough to find a $*$-invariant graded subalgebra $FC_p$. Let $h \in H$ be any nontrivial element. If $h\tau(h)=1$, then necessarily $\tau(h)=h^{-1}$. Write $|h|=pm$, where $p$ is 
a prime number. In this case, $g := h^{m}$ has order $p$ and satisfies $\tau(g)\in\{g, g^{-1}\}$. Hence, if $C_p=\langle g\rangle$, the algebra 
$FC_p \subseteq F^\sigma H$ is a $(G,*)$-subalgebra isomorphic to either $FC_{p,\tau}^g$ or $FC_{2,*}^g$. This proves the claim.

On the other hand, if $h\tau(h) \ne 1$, then 
$x := h\tau(h)\in H$ is a nontrivial element such that $\tau(x) = x$. In this case, 
the proof proceeds analogously to the previous paragraph.
\end{proof}




Before we present the next result, let us define the following set of $(G,*)$-algebras:\[
\mathcal{I}= \{\, FC_{2,*},\, M_{\rho,\tau}^{g},\, FC_{p,\tau}^{h}
\mid g,h\in G,\; p \text{ prime dividing } |G|,\; |h|=p,\; 
\tau(h)\in \{h,h^{-1}\} \,\},
\]
Moreover, if $|G|$ is even, we also add to $\mathcal{I}$ the $(G,*)$-algebras $FC_{2,*}^s$, for all $s\in G$ with $|s|=2$ and $\tau(s)=s$.

\begin{lemma}\label{starsimples}
Let $\mathcal{V}$ be a variety of $(G,*)$-algebras. If $Q\notin \mathcal{V}$, for all $Q\in \mathcal{I}$, then any $(G,*)$-simple algebra in $\mathcal{V}$ is isomorphic to $F$.
\end{lemma}
\begin{proof}
By Proposition~\ref{wedderburnmalcev}.(2) and Lemma~\ref{simplesgraded}, we have either 
\[
A \cong M_n(F^{\sigma}H)
\quad \text{or} \quad
A \cong \big(M_n(F^{\sigma}H) \oplus (M_n(F^{\sigma}H))^{\operatorname{op}}, \mathrm{ex}\big),
\]
where \(H\) is a subgroup of \(G\) and \(\sigma\) is a \(2\)-cocycle of \(H\).

In the latter case, the algebra \(D = F \oplus F\) endowed with the exchange involution and trivial grading is a \((G,*)\)-subalgebra of \(A\), which is isomorphic to \(FC_{2,*}\) as a \((G,*)\)-algebra. Hence, \(FC_{2,*} \in \mathcal{V}\), a contradiction. Therefore, we must have \(A \cong M_n(F^{\sigma}H)\) for some \(n \geq 1\). 

Assume that $n\ge2$. Since \(FC_{2,*}, M_{\rho,\tau}^{1} \notin \mathcal{V}\), it follows from \cite[Theorem~4.7]{GM} that \(\mathcal{V}\) has polynomial growth of the sequence of \(*\)-codimensions. As an ordinary algebra, $M_n(F^\sigma H)$ is a semissimple algebra. Since $n>1$, it contains a copy of $M_2(F)$. Hence, from Proposition \ref{desigualdades},
$$
c_n(\mathrm{M}_2({F}))\le c_n(A)\le c_n^*(A)\le c_n^*(\mathcal{V}),
$$
a contradiction. Thus, $n=1$ and $A\cong F^\sigma H$.

We now notice that if $H$ is non-trivial then, by Lemma \ref{subalgebraofA}, $A$ contains a $(G,*)$-subalgebra isomorphic to $FC_{p,\tau}^{h}$ or $FC_{2,*}^h$ (in case $|h|=2$), for some $h\in H$ of prime order $p$ generating $C_p$, $\tau(h)\in \{h,h^{-1}\}$. This contradiction shows that $H$ must be $\{1\}$ and so $A\cong F$.

\end{proof}

\begin{lemma} \label{mgtaug}
    Let $A$ be a finite-dimensional $(G,*)$-algebra, where $*$ is a $\tau$-involution. Consider $A=A_1\oplus \cdots \oplus A_n+J(A)$ a Wedderburn-Malcev decomposition of $A$ as a $(G,*)$-algebra.  If there exist $i, k\in \{1,\ldots, n\}$ with $i\neq k$ and $g\in G$ such that $A_iJ(A)_gA_k\neq \{0\}$ then $M_{\rho,\tau}^{g}\in \textnormal{var}^\sharp (A)$.
\end{lemma}

\begin{proof}
    Assume that $A_i J(A)_g A_k = \{0\}$ for some $i \neq k$.  Denote by $e_1$ and $e_2$ the unity elements of $A_i$ and $A_k$, respectively, and take $j_g \in J(A)_g$ such that $e_1 j_g e_2 \neq 0$. 

Let $m$ be the largest integer such that $e_1 J e_2 \subseteq J^m$ and consider the quotient $(G,*)$-algebra $A' = A / J^{m+1}$.

Define $R$ the $(G,*)$-subalgebra of $A'$ generated by $\overline{e_1}, \ \overline{e_2}, \ \overline{e_1 j_g e_2} \ \text{and} \ \overline{e_2 j_g^* e_1}.$ Since $e_1$ and $e_2$ are orthogonal idempotents and $$\overline{e_1 j_g e_2} \cdot \overline{e_2 j_g^* e_1} = \overline{e_2 j_g^* e_1} \cdot  \overline{e_1 j_g e_2} = \overline{0},$$ it follows that $R$ is linearly generated by the elements 
$\overline{e_1}$, $\overline{e_2}$, $\overline{e_1 j_g e_2}$ and $\overline{e_2 j_g^* e_1}$.

Consider the map $\varphi : R \longrightarrow M_{\rho,\tau}^{g}$ defined by
\[
\overline{e_1} \mapsto e_{11} + e_{22}, \quad 
\overline{e_2} \mapsto e_{22} + e_{33}, \quad 
\overline{e_1 j_g e_2} \mapsto e_{12}, \quad 
\overline{e_2 j_g^* e_1} \mapsto e_{34}.
\]
It is straightforward to verify that $\varphi$ defines an isomorphism of $(G,*)$-algebras. Therefore, $M_{\rho,\tau}^{g}\in \textnormal{var}^\sharp (A)$.
\end{proof}

Before we present the next lemma, we first notice that for each polynomial $${f}(x_{1,1},x_{1,1}^*,
\ldots , x_{n,1},x_{n,1}^*,  \ldots , x_{1,g_k},x_{1,g_k}^*,\ldots , x_{n,g_k},x_{n,g_k}^* )\in P_n^\sharp $$ we can associate a multilinear map ${f}': A^n\rightarrow  A$ given by
$${f}'(a_1, \ldots, a_n)= f(a_{1,1},a_{1,1}^*, \ldots, a_{n,1}, a_{n,1}^*, \ldots, a_{1,g_k},a_{1,g_k}^*,\ldots , a_{n,g_k}, a_{n,g_k}^*),$$ where $a_1= \sum_{i=1}^k a_{1,g_i}, \ldots, a_n= \sum_{i=1}^k a_{n,g_i}$. 

Note that ${f}'=0$ if and only if $f$ is a $(G,*)$-identity of $A$. Therefore, we can embed $P_n^\sharp (A)$ into the space $\mathcal{L}_n(A,A)$ of multilinear maps from $A^n$ to $A$ via the map $\psi(f)=f'$ defined above.

\begin{lemma} \label{polynomiallybounded}
     Let $A$ be a finite-dimensional $(G,*)$-algebra over a field $F$ of characteristic zero, where $*$ is a $\tau$-involution. Consider $A=A_1\oplus \cdots \oplus A_m+J(A)$ a Wedderburn-Malcev decomposition of $A$ as a $(G,*)$-algebra. If $A_i\cong F$ and $A_iJ(A)A_k=\{0\}$, for all $i,k\in\{1,\ldots, m\}$, $i\neq k$, then $A$ has polynomial growth of the sequence of $(G,*)$-codimensions.

\end{lemma}

\begin{proof}

Assume that $s$ is the nilpotency index of the Jacobson radical of $A$ and consider $\mathcal{B}$ a basis of $A$ given by the union of basis of the simple components of $A$ and a basis of the Jacobson radical. 

Let $\mathcal{T}$ be the set of all $n$-tuples $(a_1, \ldots, a_n)$ of elements of $\mathcal{B}$ such that either $a_{i_1},\ldots, a_{i_t}\in J$, for some $t\geq s$ and distinct elements $i_{1},\ldots ,i_t\in \{1, \ldots, n\}$, or $a_{i_1}\in A_k$ and $a_{i_2}\in A_\ell$ for some $k\neq\ell$. Define
$$\mathrm{U}=\{f\in \mathcal{L}_n(A,A)\mid f(a_1, \ldots, a_n)=0\mbox{ if }(a_1, \ldots, a_n)\in \mathcal{T}\}\subseteq \mathcal{L}_n(A,A).$$

Since $A_iJA_k=\{0\}$, for all $i,j\in\{1,\ldots, m\}$, $i\neq k$, and $s$ is the smallest integers such that $J^s=\{0\}$, by the previous discussion $P_n^\sharp(A)$ can embed into $\mathrm{U}$. In order to give an upper bound for $c_n^\sharp(A)=\dim_F P_n^\sharp(A)$, we compute the dimension of the space $\mathrm{U}$.  A combinatorial argument proves that $$c_n^\sharp(A)\leq \dim_F\mathrm{U}=  \displaystyle \dim_F (A) \cdot m\sum_{k=0}^{s-1} \binom{n}{k}(\dim_F (J))^k.$$

Therefore, $c_n^\sharp(A)$ is polynomially bounded and the proof is complete.
\end{proof}

Before presenting the main result, we recall the definition of the set $\mathcal{I}$ introduced earlier. 
This set consists of the following $(G,*)$-algebras:
\[
\mathcal{I}= \{\, FC_{2,*},\, M_{\rho,\tau}^{g},\, FC_{p,\tau}^{h}
\mid g,h\in G,\; p \text{ prime dividing } |G|,\; |h|=p,\; 
\tau(h)\in \{h,h^{-1}\} \,\}.
\]
Moreover, if $|G|$ is even, we also include in $\mathcal{I}$ the $(G,*)$-algebras 
$FC_{2,*}^s$, for all $s\in G$ such that $|s|=2$ and $\tau(s)=s$.

\begin{theorem} \label{main_thm}
		Let $\mathcal{V}$ be a $(G,*)$-variety generated by a finite-dimensional $(G,*)$-algebra $A$, where $*$ is a $\tau$-involution on $A$. Then $\mathcal{V}$ has polynomial growth if and only if each $Q\notin \textnormal{var}^\sharp(A)$, for all $Q\in \mathcal{I}$.
        
        	\end{theorem}

\begin{proof}

Recall that, by Remark \ref{exponentialgrowth}, each algebra \(B \in \mathcal{I}\) generates a variety with exponential growth of the sequence of \((G,*)\)-codimensions. Hence, if \(\mathcal{V}\) has polynomial growth, 
we necessarily have \(B \notin \mathcal{V}\) for all \(B \in \mathcal{I}\).

Conversely, assume that \(B \notin \mathcal{V}\) for all \(B \in \mathcal{I}\). Let  
\[
A = B_1 \oplus \cdots \oplus B_m + J(A)
\]
be the Wedderburn-Malcev decomposition of $A$.

According to Lemma~\ref{starsimples}, we may assume that \(B_i \cong F\) for all \(i\). 
Moreover, since \(M_{\rho, \tau }^{g} \notin \textnormal{var}^\sharp(A)\) for all \(g \in G\), 
Lemma~\ref{mgtaug} ensures that 
\[
B_i J(A) B_k = \{0\}, 
\quad \text{for all } i,k \in \{1,\ldots,m\} \text{ with } i \neq k.
\]
Hence, Lemma~\ref{polynomiallybounded} implies that \(A\) has polynomial growth 
of the sequence of \((G,*)\)-codimensions.
\end{proof}

As a consequence of Proposition \ref{exponentiallubounded} and the previous theorem, we obtain the following result.

\begin{corollary}
   Let $\mathcal{V}$ be a $(G,*)$-variety generated by a finite-dimensional $(G,*)$-algebra $A$, where $*$ is a $\tau$-involution on $A$. Then the sequence $c_n^\sharp(A)$, $n \geq 1$, is either polynomially bounded or grows exponentially.\qed
\end{corollary}

Recall that two varieties $\mathcal{V}$ and $\mathcal{W}$ are said to be not comparable whenever neither is contained in the other, i.e., $\mathcal{V}\not\subseteq \mathcal{W}$ and $\mathcal{W}\not\subseteq \mathcal{V}$. Let us define $$\mathcal{V}_*=\textnormal{var}^\sharp(FC_{2,*}),\quad \mathcal{V}_{2,*}^s=\textnormal{var}^\sharp(FC_{2,*}^{s}), \quad \mathcal{V}_\tau^{g}=\textnormal{var}^\sharp (M_{\rho,\tau}^{g}), \quad \mathcal{V}_{p,\tau}^{h}=\textnormal{var}^\sharp(FC_{p,\tau}^{h}),$$ for all $g,s\in G$ with $|s|=2$ and $\tau(s)=s$ (in case $|G|$ is even) and for all prime number $p\mid |G|$ and $h\in G$ with $|h|=p$ such that $\tau(h)\in \{h,h^{-1}\}$.

Under this terminology, we have the following result.

\begin{lemma}\label{lema}
    The varieties $\mathcal{V}_*,\, \mathcal{V}_{2,*}^s,\, \mathcal{V}^{g}_\tau$ and $ \mathcal{V}^{h}_{p,\tau }$, are not comparable, for any $g,s\in G$ with $|s|=2$ and $\tau(s)=s$ (in case $|G|$ is even) and any prime number $p\mid |G|$ and $h\in G$ with $|h|=p$ such that $\tau(h)\in \{ h,h^{-1}\}$.
\end{lemma}

\begin{proof}
    The supports of $M_{\rho,\tau}^g$ and $M_{\rho,\tau}^h$ are equal if and only if 
$g \in \{h, \tau(h)\}$; equivalently, $M_{\rho,\tau}^g \cong M_{\rho,\tau}^h$. 
Thus, if $M_{\rho,\tau}^g \not\cong M_{\rho,\tau}^h$, then the corresponding 
$(G,*)$-varieties of identities are not comparable.

The variety $\mathcal{V}_\tau^g$ does not satisfy $[x_{1,1},x_{2,g}]$, which is 
satisfied by all the other varieties. Let $A$ be a $(G,*)$-algebra generating any 
of the remaining varieties. If there exists 
$h \in \mathrm{supp}\,A \setminus \mathrm{supp}\,M_{\rho,\tau}^{g}$, then the 
polynomial $x_{1,h}$ is an identity of $\mathcal{V}_\tau^g$ but not of $A$. 
Otherwise, if $g \ne 1$, the polynomial $x_{1,g}^2$ is satisfied by 
$\mathcal{V}_\tau^g$ but not by $A$. If $g=1$, then 
$\mathrm{supp}\,A = \{1\}$, i.e., $A = FC_{2,*}$. In this case, 
$(x_{1,1} - x_{1,1}^*)^2$ is satisfied by $M_{\rho,\tau}^{1}$ but not by $A$.

Hence, each $\mathcal{V}_\tau^g$ is not comparable with any of the other distinct 
varieties in the list.

Now, $FC_{p,\tau}^g \cong FC_{p',\tau}^{g'}$ if and only if $p=p'$, 
$\langle g \rangle = \langle g' \rangle$, and $\tau(g)=\tau(g')$. In particular, 
if $p>2$, then $(G,*)$-isomorphism between them is equivalent to equality of 
supports. Therefore, if $p>2$, the varieties $\mathcal{V}_{p,\tau}^g$ and 
$\mathcal{V}_*$ are not comparable with any distinct variety in the list.

It remains to show that $\mathcal{V}_{2,*}^g$ and $\mathcal{V}_{2,\tau}^g$ are not 
comparable, where $|g| = 2$ and $\tau(g) = g$. This follows by considering the 
polynomials $x_{1,g} - x_{1,g}^*$ and $x_{1,g} + x_{1,g}^*$.
\end{proof}

\begin{corollary}
 Let $\mathcal{V}$ be a $(G,*)$-variety generated by a finite-dimensional $(G,*)$-algebra $A$ over a field $F$ of characteristic zero (not necessarily algebraically closed), where $*$ is a $\tau$-involution. Then $\mathcal{V}$ has almost polynomial growth if and only if $A$ is $T_G^*$-equivalent to one of the following algebras:   

\begin{enumerate}
\renewcommand{\labelenumi}{(\arabic{enumi})}
    \item $FC_{2,*}$;
    \item $M_{\rho,\tau}^{g}$, for some $g\in G$;
      \item $FC_{2,*}^s$, for some $s\in G$ with $|s|=2$ and $\tau(s)=s$ (in case $|G|$ is even);
    \item $FC_{p,\tau}^{h}$, for some prime number $p\mid |G|$ and $h\in G$ with $|h|=p$ such that $\tau(h) \in \{h,h^{-1}\}$.
\end{enumerate}

\end{corollary}

\begin{proof}

In fact, assume that $A$ has almost polynomial growth of the sequence of $(G,*)$-codimensions. According to Theorem~\ref{main_thm}, there exists an algebra $B \in \mathcal{I}$ such that $B \in \textnormal{var}^\sharp(A)$. Since any proper subvariety of $\mathcal{V}$ has polynomial growth, we must have $\mathcal{V} = \textnormal{var}^\sharp(B)$.

Conversely, assume that $\mathcal{V}$ is generated by some algebra from $\mathcal{I}$, and let $\mathcal{W}$ be a proper subvariety of $\mathcal{V}$. According to Lemma~\ref{lema}, we have $Q \notin \mathcal{W}$ for all $Q \in \mathcal{I}$. By Theorem~\ref{main_thm}, it follows that $\mathcal{W}$ has polynomial growth.

\end{proof}


\begin{thebibliography}{999}
\bibitem{BaSehZa} Y. Bahturin, S. Sehgal and M. Zaicev, \emph{Group gradings on associative algebras}, J. Algebra \textbf{241} (2001) 677-698.

\bibitem{Bahturin} Y. Bahturin and M. Zaicev, \emph{Group gradings on matrix algebras}, Canad. Math. Bull. {\bf 45} (2002) 499–508.

\bibitem{BSZ} Y. Bahturin, I. Shestakov and M. Zaicev, \emph{Gradings on simple Jordan and Lie algebras}, J. Algebra \textbf{283} (2005) 849-868.


\bibitem{Mello} T.~de Mello, \emph{Homogeneous involutions on upper triangular matrices}, Arch. Math. (Basel) \textbf{118} (2022) 365-374.

\bibitem{MY} T.~de Mello and F.~Yasumura, \emph{On star-homogeneous-graded polynomial identities of upper triangular matrices over an arbitrary field}, J. Algebra \textbf{663} (2025) 652-671.

\bibitem{EK13} A.~Elduque and M.~Kochetov, \emph{Gradings on simple Lie algebras}. Math. Surveys Monogr., 189. American Mathematical Society, Providence, RI Atlantic Association for Research in the Mathematical Sciences (AARMS), Halifax, NS, 2013.

\bibitem{Mello1} L.~F.~Fonseca and T.~de Mello, \emph{Degree-inverting involutions on matrix algebras}, Linear Multilinear Algebra \textbf{66} (2018) 1104-1120.

\bibitem{GM} A. Giambruno and S. Mishchenko. \textit{ On star-varieties with almost polynomial growth}. Algebra Colloq. {\bf 8} (2001) 33-42.

\bibitem{GM2} A. Giambruno and S. Mishchenko,  \textit{ Polynomial growth of the $*$-codimensions and Young diagrams}, Comm. Algebra {\bf 29} (2001) 277-284

\bibitem{GReg} A. Giambruno and A. Regev, \textit{ Wreath products and P.I. algebras}, J. Pure Appl. Algebra {\bf 35} (1985) 133-149.

\bibitem{GiZai} A. Giambruno and M. Zaicev, 
\textit{Exponential codimension growth of PI-algebras: an exact estimate}, Adv. Math. {\bf{142}} (1999) 221-243.
  


\bibitem{Gord} A. S. Gordienko, \textit{Co-stability of radicals and its applications to PI-theory}, Algebra Colloq.  \textbf{23} (2016) 481-492.

\bibitem{Hazrat} R.~Hazrat. Graded Rings and Graded Grothendieck Groups. Cambridge: Cambridge University Press; 2016.

\bibitem{Kem} A. R. Kemer, \textit{Varieties of finite rank}, Proc. 15th All the Union Algebraic Conf., Krasnoyarsh.  {\bf 2} (1979) (in Russian).

  \bibitem{MV} S. Mishchenko and A. Valenti, \textit {A star-variety with almost polynomial growth}, J. Algebra {\bf 223} (2000) 66-84.


\bibitem{Lorena} L. M. Oliveira, R. B. dos Santos and A. C. Vieira, \textit{ Varieties of group graded algebras with graded involution of almost polynomial growth}, Algebr. Represent. Theory {\bf  26} (2023) 663-677.

\bibitem{PZ89} J.~Patera and H.~Zassenhaus, \textit{On Lie gradings I}, Linear Algebra Appl. \textbf{112} (1989) 87-159.

\bibitem{RG} A. Regev, \textit{Existence of identities in $A\otimes B$}, Israel J. Math. {\bf{11}} (1972) 131-152.


  \bibitem{Valenti} A. Valenti, \textit{Group graded algebras and almost polynomial growth}, J. Algebra {\bf 334} (2011) 247-254.
  
\bibitem{Y} F.~Yasumura, \emph{Homogeneous involution on graded division algebras and their polynomial identities}, J. Algebra Appl. \textbf{23} (2024).
	
\end{thebibliography}
\end{document}